\documentclass[11pt]{amsart}
\usepackage{amsmath}
  \usepackage{paralist}
  \usepackage{graphics} 
  \usepackage{epsfig} 
\usepackage{graphicx}  \usepackage{epstopdf}
 \usepackage[colorlinks=true]{hyperref}
\hypersetup{urlcolor=blue, citecolor=red}
\newtheorem{theorem}{Theorem}[section]

\theoremstyle{definition}
\newtheorem{definition}[theorem]{Definition}
\newtheorem{remark}{Remark}
\newtheorem{example}{Example}

\title[On ${\mathcal L}^1$ limit solutions] 
      {On ${\mathcal L}^1$ limit solutions in impulsive control }
\author[Monica Motta  and Caterina Sartori]{}


 \email{motta@math.unipd.it}
 \email{sartori@math.unipd.it}

\thanks{This research is partially supported by the  INdAM-GNAMPA Project 2017 "Optimal impulsive control: higher order necessary conditions and gap phenomena";  and by the Padova University grant PRAT 2015 ``Control
of dynamics with reactive constraints''}

\begin{document}
\maketitle

\centerline{\scshape Monica Motta }
\medskip
{\footnotesize
 \centerline{Dipartimento di Matematica ``Tullio Levi-Civita'', }
   \centerline{Universit\`a di Padova}
   \centerline{ Via Trieste, 63, Padova  35121, Italy}
} 

\medskip

\centerline{\scshape Caterina Sartori}
\medskip
{\footnotesize
 \centerline{ Dipartimento di Matematica ``Tullio Levi-Civita'',}
   \centerline{Universit\`a di Padova}
   \centerline{Via Trieste, 63, Padova  35121, Italy}
}

\bigskip

\begin{abstract}
We consider a nonlinear control system   depending on two controls $u$ and $v$,  with  dynamics  affine in the  (unbounded) derivative of  $u$, and   $v$ appearing initially only in the drift term.  Recently, motivated by applications to optimization problems lacking coercivity,  \cite{AR} proposed  a notion of generalized solution $x$  for  this system, called {\it limit solution,}    associated to  measurable $u$  and $v$, and   with    $u$ of possibly unbounded  variation in $[0,T]$. As shown in \cite{AR}, when $u$ and $x$ have bounded variation,  such a solution (called in this case BV simple limit solution)  coincides with the most used graph completion solution (see e.g. \cite{BP}).  This correspondence  has been extended in \cite{MSuv}  to BV$_{loc}$   inputs $u$  and trajectories (with bounded variation just on any $[0,t]$ with $t<T$). Starting  with   an example of optimal control where the minimum does not exist in the class of limit solutions,  we propose  a notion of {\it extended limit solution} $x$, for which such a minimum exists.   As a first result, we prove that extended and original limit solutions coincide in the   special cases of BV and BV$_{loc}$ inputs $u$ (and solutions).  Then we consider dynamics where the ordinary control $v$  also appears  in the non-drift terms.   For the associated system we prove  that,   in the BV case,  extended  limit solutions   coincide  with   graph completion solutions. 
\end{abstract}

\section{Introduction}
We consider a control system of the form  
\begin{equation}\label{E}
 \dot x (t)= g_0(x(t),u(t),v(t)) +  \sum_{i=1}^m  {g}_i(x(t),u(t))\dot u_i(t) \quad \text{a.e. $t\in[0,T]$,} 
 \end{equation}
 \begin{equation}\label{EO}
 x(0)=\bar x_0,\quad u(0)=\bar u_0,
\end{equation}
 where $x\in {\mathbb{R}}^n$, $(u(t),v(t))\in U\times V$   and $U$, $V$ are compact sets.    
 \noindent System  \eqref{E} is a so-called {\it impulsive} control system,   where a solution $x$ can be  provided by the usual  Carath\'eodory solution only if $u$ is an absolutely continuous  control. For less regular $u$,   several concepts of impulsive solution have been introduced in the literature, either for  {\it commutative  systems, } where the Lie brackets $[({\bf e}_i,g_i),({\bf e}_j,g_j)]=0$ for all $i,j=1,\dots,m$  (see e.g. \cite{BR1}), or assuming $u$ (and $x$) to be functions of bounded variation, when the Lie Algebra is non trivial.  These  solutions  are described by different authors in fairly equivalent ways, and we will refer to them as graph completion solutions, since they are obtained  by completing the graph of $u$ (see e.g. \cite{BR},  \cite{MR},  \cite{SV}, \cite{MiRu}, \cite{WS}, \cite{AKP}, \cite{GS}, \cite{KDPS}).  In the less studied non commutative case  with measurable controls $u$ of unbounded variation, let us  mention  \cite{BR2}, \cite{LQ},  and the  definition of {\it limit solution} due to  \cite{AR}.  In the special case of  {\it BV simple limit solutions,} in which $u$ and $x$ are of bounded variation, in  \cite{AR} the authors showed that  any limit solution is a   graph completion solution and vice-versa (see Definitions   \ref{ls},  \ref{GCloc}, \ref{GClocS}).  This is an important result, since, on the one hand,  graph completion solutions have a simple explicit representation formula, not available for general  limit solutions. On the other hand, it proves that (pointwisely defined) graph completion solutions  are well-posed, in the sense that they coincide with all  and only pointwise limits of classical  solutions.   In  \cite{MSuv} we  extended  such a result to a case of unbounded variation, by  introducing  graph completion solutions associated to BV$_{loc}$ inputs $u$ (and trajectories) and  we proved  that they coincide with a special subset of simple limit solutions,   the BV$_{loc}$ simple limit solutions (see Definition \ref{lsl}).
 
In this paper we analyse the concept of limit solution and, starting from  an example in optimal control for  which the infimum over limit solutions is not a minimum,   we introduce a notion of {\it extended limit solution},  where such a  minimum does exist.
 As a first result, in Theorem \ref{Pvk} we prove that this new definition coincides with the original one in the special  cases of  of BV simple or  BV$_{loc}$ simple  limit solutions (see Definitions \ref{ls}, \ref{lsl}).  As a consequence,  all the results available for  these two classes of limit solutions are still  valid for their extended counterpart. 
 
  Furthermore, we investigate  control systems of the form
  \begin{equation}\label{Evintro}
 \dot x (t)= g_0(x(t),u(t),v(t)) +  \sum_{i=1}^m  {g}_i(x(t),u(t),v(t))\dot u_i(t) \quad \text{a.e. $t\in[0,T]$,} 
 \end{equation}
  where  all the $g_0,g_1,\dots,g_m$  depend  on the control $v$.  The definition of limit solution for \eqref{Evintro}  was left as an open problem in \cite{AR}.  Indeed, our  notion of extended limit solution can be adapted to this case, allowing us to show,   in Theorem \ref{Th2},   that extended BV simple limit solution  and  graph completion solutions to  \eqref{Evintro}, \eqref{EO} coincide.  This result   extends to   system  \eqref{Evintro} the analogous   \cite[Thm. 4.2]{AR} regarding  \eqref{E}.  As remarked in \cite{AR},   already  when $u$ (and $x$) has bounded variation, the dependence of $g_1,\dots,g_m$ on $v$  is much more critical than   just  the $v$-dependence of $g_0$,  in that a simultaneous jump of  $u$  and $v$  makes  the determination of the corresponding jump of $x$ quite delicate.
   
 The precise definitions of limit solution and extended limit solution  will be given in Sections \ref{Prel},  \ref{mr}. Here we just point out that the notion of  limit solution involves  a  control  $v$  which is measurable,  while  the control  $u$ and the corresponding solution $x$ are pointwisely defined and belong to the set ${\mathcal L}^1$ of the everywhere defined integrable functions. Let us  describe  a special case of  extended limit solution.  An {\it extended  simple limit solution}  $x$ to \eqref{E}, \eqref{EO}  associated to $(u,v)$,  is  the  pointwise  limit of a sequence of classical trajectories $(x_k)$ to \eqref{E}, \eqref{EO}, corresponding to controls $(u_k,v_k)$  with  $u_k$ absolutely continuous and pointwisely converging to $u$  and  $v_k\to v$ in $L^1$-norm (see Definition \ref{edsdef}).  We recall that  a {\it simple limit solution} $x$  is instead defined in \cite{AR} as  the  pointwise  limit of a sequence of classical trajectories associated to controls $(u_k,v)$  with  $u_k$ as above and $v$ fixed (see Definition \ref{ls}).  Our extension  is motivated by the observation  that in optimal control problems  minimizing sequences $(x_k, u_k,v_k)$ with absolutely continuous inputs $u_k$, might converge to a map which is not a limit solution. Precisely, in Example \ref{Ex3} we have that the infimum value of an optimal control problem over limit solutions and extended limit solutions is the same, but it is a minimum only within  the larger class of extended limit solutions. The   two infima may be  actually different, as shown in  Example  \ref{Ex3bis}.

The need of considering generalized solutions to   \eqref{E} or  \eqref{Evintro} and   \eqref{EO}, associated to discontinuous $u$  comes, for instance,  from  optimal control, where, in absence of coercivity assumptions,  it is reasonable to expect the existence of   optimal solutions only in some enlarged class.  The impulsive control theory,  studied since the 50s,  received  in the last years a renewed attention because of the increasing number  of applications in different  fields,  from Lagrangian mechanics with moving constraints \cite{BR0}, \cite{BP},   or impactively blockable degrees of freedom  \cite{Y}, \cite{GoS}, to  alternative models for hybrid systems \cite{A}, \cite{FGP}, \cite{KT},   \cite{HCN},  just to give some examples. These applications set new problems also from the theoretical point of view, in particular since they lead to consider control systems nonlinear in the state variable  like   \eqref{E}  or  \eqref{Evintro},  and    various types of constraints.

{\vskip 0.2 truecm}
The paper is organized as follows. 
We end this section with some notation and  the precise  assumptions.  In Section \ref{Sex} we present two examples that motivate the notions of extended limit solutions, which we propose in Section   \ref{mr}.   Section  \ref{Prel}, is devoted to recall the original concepts of limit  solution due to \cite{AR} and the recent definition of BV$_{loc}$ limit solution introduced in \cite{MSuv}.   In Theorem  \ref{Pvk} of Section \ref{mr} we prove that original and extended   BVS limit solutions and   BV$_{loc}$S limit solutions, respectively,  coincide.  In Section \ref{Sfex} we introduce the $v$-dependent control system \eqref{Evintro} and in Theorem \ref{Th2} we establish that a map $x$ is an extended BVS limit solution to  \eqref{Evintro}, \eqref{EO} if and only if it is a graph completion solution.

\subsection{Notation}  Let  $E\subset {\mathbb{R}}^N$.  Given  $T>0$,  let  \newline $AC([0,T],E):=\{f:[0,T]\to E, f\,{\rm absolutely \,continuous}\}$,  \newline     $BV([0,T], E):=\{f:[0,T]\to E:\, Var_{[0,T]}(f)<+\infty\}$,  \newline
where   $Var_{[0,T]}(f)$   denotes the (total) variation of $f$ in $[0,T]$, and \newline
 $BV_{loc}([0,T[,E):= \{f\in BV([0,t],E)   \, \forall t<T , \   \lim_{t\to T}Var_{[0,t]}[f]\le+\infty \}.$

\noindent We  use ${\mathcal L}^1([0,T], E)$ to denote the set of the everywhere defined integrable functions on $[0,T]$ with values in $E$, while  $L^1([0,T], E)$ is its usual quotient space with respect to the Lebesgue measure.  When no confusion on the codomain may arise, we omit it and  write, for instance, $AC(T)$ in place of $AC([0,T],E)$.  Let us set  ${\mathbb{R}}_+:=[0,+\infty[$ and  call  {\it modulus } (of continuity) any increasing, continuous function $\omega:{\mathbb{R}}_+\to{\mathbb{R}}_+$ such that $\omega(0)=0$ and  $\omega(r)>0$ for every $r>0$.
 
\noindent For any control $(u,v)\in {AC}(T)\times  {L^1}(T)$ with $u(0)=\bar u_0$, we let
$$
x=x[\bar x_0,\bar u_0,u,v]
$$
denote the (unique) Carath\'eodory solution to   (\ref{E})--(\ref{EO}), defined on $[0,T]$.  We will say that such $(u,v)$ and $x$ are {\it regular}.

\subsection{Assumptions} Let us recall   the so--called Whitney property (see \cite{W}). 
\begin{definition}[Whitney property]\label{W}  
A  compact subset  $U\subset {\mathbb{R}}^m$ has the Whitney property if there is some $C\ge1$ such that for every pair $(u_1 , u_2 ) \in U \times U$, there exists an absolutely continuous path $\tilde u: [0, 1] \to U$ verifying
\begin{equation}\label{hCA}
\tilde u(0)=u_1, \quad \tilde u(1)=u_2, \quad Var[\tilde u]\le C|u_1-u_2|.
\end{equation}
\end{definition}
For instance,  compact, star-shaped sets enjoy the Whitney property. 

{\vskip 0.2 truecm}
Throughout the paper we assume the following hypotheses:
\begin{itemize} \item[{\bf (H0)}] 
{\it \begin{itemize}
\item[{\rm (i)}] the sets $U \subset {\mathbb{R}}^m$,  $V \subset {\mathbb{R}}^l$ are compact and $U$ has the Whitney property;
 \item[ {\rm (ii)}] the control vector field $g_0: {\mathbb{R}}^n \times U \times V\to{\mathbb{R}}^n$  is continuous and, moreover,
 $ (x,u) \mapsto  g_0(x,u,v)$
 is locally Lipschitz on ${\mathbb{R}}^n \times U$ uniformly in $v \in V$;
  \item[ {\rm (iii)}] for each $i=1,\dots,m$  the control vector field  $g_i: {\mathbb{R}}^n \times U\to{\mathbb{R}}^n$  is locally Lipschitz continuous;
  \item[ {\rm (iv)}] there exists $M>0$ such that
  $$
|g_0(x,u,v)|, \, |g_1(x,u)|,\dots, \, |g_m(x,u)| \leq  M(1+|(x,u)|),$$
  for every $(x,u,v) \in {\mathbb{R}}^n\times U\times V$.
\end{itemize}}
  \end{itemize}

\section{Examples}\label{Sex}
This section is devoted to motivate,  by means of two simple examples,  the need of enlarging the class of limit solutions, introducing  a notion of {\it extended} limit solution. Precisely, in Example \ref{Ex3} we exhibit  an optimal control problem where the infimum value  over limit solutions and extended limit solutions is the same, but the minimum is achieved only within  the larger class of extended limit solutions. In Example \ref{Ex3bis} we present a minimum problem where there is  a gap between the  infimum over limit solutions and extended limit solutions and a gap between the  infimum over regular solutions and  limit solutions.
 
\noindent These phenomena may happen since in  both examples {\it any } regular minimizing control sequence $(u_k,v_k)$ verifies $\displaystyle\lim_{k\to+\infty}$Var$(u_k)=+\infty$.

 {\vskip 0.2 truecm}  
\begin{example}\label{Ex3} {\rm 
Let us consider  the control system in ${\mathbb{R}}^4$,
\begin{equation}\label{din3}
\dot x =  g_0(x)\, v+g_1(x)\,\dot u_1+g_2(x)\,\dot u_2 \ \ \text{a.e. $t\in[0,2\pi]$,}  \quad |u|,\,|v|\le 1, 
\end{equation}
with 
$$\begin{array}{llll}
g_0(x):=\eta(x) (0,0,0,v)^T\\
g_1(x):=\eta(x)(1,0,x_3x_2,{-x_4x_2})^T\\
g_2(x) := \eta(x)(0,1,-x_3x_1,{x_4x_1})^T \end{array}
$$
($\eta$ is a {cut-off function,} sufficient to guarantee the sublinearity hypothesis on the dynamics) 
and initial condition 
$$
(x,u)(0):= (\bar x_0,\bar u_0)=((0,0,1,0) ,(0,0)).
$$
Let us  introduce  the Bolza optimization problem 
$$
\displaystyle\inf_{(x,u,v)}\, J(x,u,v),
$$
 where
$$
J(x,u,v):= \int_0^{2\pi}(|u(t)|+|v(t)|)\,dt+(2\pi-x_4(2\pi))^2.
$$
We now construct a minimizing sequence $(x_k, u_k,v_k)$ within the class of regular trajectory-control pairs.  For every $k$,   let us set, for $t\in[0,2\pi]$, 
\begin{equation}\label{uk}
(u_k,v_k)(t){:=} \left(\frac{1}{\sqrt[3] k} \big( \cos(kt)-1,\sin(kt) \big)\chi_{_{[2\pi/k,2\pi]}}(t)\,,\,  k\,e^{-2\pi\sqrt[3] k}\chi_{_{[0,2\pi/k]}}(t)\right).
\end{equation} 
The  corresponding solution $x_k:=x[\bar x_0,\bar u_0,u_k,v_k]$ is given,  for $t\in[0,2\pi]$,  by
 $$
 \left\{\begin{array}{l}
 x_{1_k}(t)= u_{1_k}(t), \\ [1.5ex]
 x_{2_k}(t)= u_{2_k}(t),  \\ [1.5ex]
 x_{3_k}(t)= \chi_{_{[0,2\pi/k[}}(t)+e^{-\sqrt[3]{k}\left(t-\frac{\sin(kt)}{k}-\frac{2\pi}{k}\right)}\chi_{_{[2\pi/k,2\pi]}}(t), \\ [1.5ex]
 x_{4_k}(t)= 
 k\,e^{-2\pi\sqrt[3] k}\,t\chi_{_{[0,2\pi/k[}}(t)+2\pi\,e^{ \sqrt[3] k\left(t-2\pi-\frac{\sin(kt)}{k}-\frac{2\pi}{k}\right)} \chi_{_{[2\pi/k,2\pi]}}(t).\end{array}\right.
$$ 
One has that 
$$
\displaystyle\lim_{k\to+\infty} \, J(x_k,u_k,v_k)=0,
$$ 
so that the infimum of the cost over  regular trajectory-control pairs turns out to be $0$. Clearly, this is  not a minimum, since  the unique optimal control must be $u\equiv0$ and $v=0$ a.e.,  whose associated Charath\'eodory solution to \eqref{din3} gives a cost equal to $4\pi^2$.  A minimum can be reached  only over some enlarged set of generalized controls and solutions. Notice that  
$$
\displaystyle\lim_{k\to+\infty}\,u_k(t)=0 \ \ \forall t\in[0,2\pi], \quad \lim_{k\to+\infty}\|v_k-v\|_{_{L^1(2\pi)}}=0.
$$ 
Hence if we  define as {\it extended limit solution} to  \eqref{din3} associated to the control $(u,v)=(0,0)$ a.e., the  limit function
\begin{equation}\label{esls}
\displaystyle x(t):=\lim_{k\to+\infty} x_k(t)=(0,0,1,0)\,\chi_{_{\{t=0\}}}(t)+(0,0,0,0)\chi_{_{]0,2\pi[}}(t)+ (0,0,0,2\pi)
\end{equation}
for $t\in[0,2\pi]$,   we obtain  
$$
J(x,0,0)=0.
$$
Therefore in the class of extended limit solutions  the minimum does exist (see Definition \ref{edsdef}).

 {\vskip 0.2 truecm}   
Let us point out that $x$ is {\it not} a {\it limit solution}  as defined in \cite{AR}, because of the varying $v_k$ (see Definition \ref{ls}). Indeed, as already observed, the optimal control has to be $(u,v)=(0,0)$ a.e., but any sequence $\tilde x_k:=x[\bar x_0,\bar u_0,\tilde u_k,0]$ associated to an arbitrary sequence $(\tilde u_k)$ pointwisely converging to $0$, verifies
$$
 \tilde x_{4_k}\equiv0 \quad\text{for every $k$,}
$$
so that $J(\tilde x_k,\tilde u_k,0)=4\pi^2$ for every $k$. Thus the minimum of the above optimization problem does not exist in the class of limit solutions. } 
\end{example}

Slightly modifying the previous example and adding some constraints, we can provide a case where the infima over regular solutions, over  limit solutions and over extended limit solutions are all different.

\begin{example}\label{Ex3bis} {\rm
Let us introduce the control system in ${\mathbb{R}}^5$, obtained by adding to \eqref{din3} the equation
$$
\dot x_5(t)=|v(t)|+|u(t)|  \quad \text{ for a.e. $t\in[0,2\pi]$,}
$$
with  initial  and end-point conditions 
$$
(x,u)(0):= (\bar x_0,\bar u_0)=((0,0,1,0,0) ,(0,0)), \quad x(2\pi)\in {\mathbb{R}}^4\times\{0\}.
$$
Let  us now set $\Psi(x):=|x_3|+|2\pi- x_4|$  for any $x\in {\mathbb{R}^5}$ and  consider the Mayer problem 
$$
\displaystyle\inf_{(x,u,v)}\, \Psi(x(2\pi)).
$$
Let us call admissible the  trajectory-control pairs satisfying the constraints.
Since only controls $(u,v)$ with  $(u,v)=0$ a.e. give rise to admissible trajectories,  the  calculations  in Example \ref{Ex3} imply that the unique  admissible  regular solution $x=x[\bar x_0,\bar u_0,0,0]$ has $(x_3,x_4)\equiv (1,0)$. Hence the infimum of the cost   over  regular solutions  is equal to $1+2\pi$. All admissible limit solutions $\tilde x$ are  pointwise limits of regular solutions $\tilde x_k:=x[\bar x_0,\bar u_0,\tilde u_k,0]$, associated to regular control sequences $(\tilde u_k)$  converging to $u=0$ (and fixed $v=0$).  Hence $\tilde x_4\equiv 0$ in any case, but  taking $\tilde u_k:=u_k$ defined by  \eqref{uk}, one has $\tilde x_3(2\pi)=1$, so that the minimum in the class of limit solutions is $\Psi(\tilde x(2\pi))=2\pi$.  Finally, the extended limit solution $x=(x_1,\dots, x_4 , x_5)=(x_1,\dots, x_4 ,0)$, where $(x_1,\dots,x_4)$  are given by \eqref{esls}, is associated to the  control $(u,v)=(0,0)$  a.e.,  verifies the constraints and has cost  $\Psi(x(2\pi))=0$. Therefore the minimum over extended limit solutions exists and is equal to $0$.   
 }
\end{example}
Let us point out that  when  there are no constraints and the cost is continuous,  by the very definition of limit solution,  the infimum value  over the different classes of solutions considered above is always the same.    The difference between the infima, as in Example \ref{Ex3bis},  is instead a  generic  situation in the presence of constraints, which are unavoidable in most   applications. In this note we do not discuss the  Lavrentiev-type gap issue, that is    the occurrence of infimum gaps   (see e.g. \cite{AMR}). Let us just observe that  in several real models, as for instance the mechanical examples in \cite{BP},  only absolutely continuous controls $u$  are implementable. In these cases, the no-gap requirement is mandatory.

\section{Definitions  and preliminary results} \label{Prel}
We start recalling  the concept   of  limit solution,   given in \cite{AR}  for vector fields $g_1,\dots, g_m$ depending on $x$ only and extended to $(x,u)$-dependent data in \cite{AMR}.   We will write ${\mathcal L}^1(T):={\mathcal L}^1([0,T],U)$  to denote the set of pointwisely defined Lebesgue integrable functions with values in $U$ and set $L^1(T):= L^1([0,T],V)$, $AC(T):= AC([0,T],U)$.
\begin{definition}[{\sc Limit solutions}]\label{ls}  
Let   $(\bar x_0,\bar u_0) \in {\mathbb{R}}^n\times U$ and let  $(u,v)\in {\mathcal L}^1(T)\times L^1(T)$ with $u(0)=\bar u_0$.
\begin{enumerate}
\item {\sc (Limit Solution)} A map $x$ belonging to ${\mathcal L}^1([0,T], {\mathbb{R}}^n)$ is  called a {\em limit solution}   of the Cauchy problem (\ref{E})-(\ref{EO}) corresponding to  $(u,v)$
 if, {\em for every $\tau \in [0,T],$}  there is a sequence of controls $(u^\tau_k)\subset  AC(T)$   such that $u^\tau_k(0)=\bar u_0$ and
\begin{itemize}
\item[{(i$_\tau$)}] the sequence $(x^\tau_k)$ of  the  Carath\'eodory solutions $x^\tau_k:=  x[\bar x_0,\bar u_0, u^\tau_k,v]$ to (\ref{E})-(\ref{EO}) is equibounded in $[0,T]$;
 {\vskip 0.2 truecm}
\item[{(ii$_\tau$)}]
$|(x^\tau_k,u^\tau_k)(\tau)- (x,u)(\tau)|+ \|(x^\tau_k,u^\tau_k)- (x,u)\|_{_{L^1(T)}}\to 0$ as $k\to+\infty$.
 {\vskip 0.2 truecm}
\end{itemize}
\item {\sc (S limit solution)}
A limit solution  $x$ is  called a  {\em simple limit solution} of (\ref{E})-(\ref{EO}), shortly   S {\em limit solution},  if the sequences $(u^\tau_k)$ can be chosen independently of $\tau.$ In this case we write $(u_k)$ to refer to the approximating sequence.
\item  {\sc (BVS  limit solution)}\label{sedsdef}
An  S  limit solution  $x$ is called a   BVS  {\em  limit  solution} of (\ref{E})-(\ref{EO})   if   the approximating inputs  $u_k$ have equibounded variation in $[0,T]$.
\end{enumerate}
\end{definition}

For a detailed discussion on the notion of limit solution we refer the reader to \cite{AR}, \cite{AMR}. Here let us just underline that, already the BVS limit solution associated to a control  $(u,v)\in {\mathcal L}^1(T)\times L^1(T)$ is not unique, unless   the system is commutative.  
Moreover,  the sets of limit solutions,  S limit solutions and  BVS  limit solutions form a decreasing sequence of sets. 

The density approach adopted in  Definition \ref{ls}  allows a unified notion of  trajectory (for commutative and non commutative systems with $u$ of possibly unbounded variation), but it does not give any explicit representation formula for the solution. In fact, such a representation exists if either the control system is commutative or  if there are a priori bounds on the variation of  the controls $u$.  In particular, in the latter case \cite{AR}  proves that  BVS  limit  solutions  coincide with graph completion solutions. The graph completion approach  is  traditionally used to study impulsive control systems with bounded variation  on $u$  (see e.g. \cite{BP} and the references therein). It provides a nice representation formula, suitable to derive, for instance  necessary and sufficient optimality conditions for several optimization problems,  both in terms of  Pontrjagin Maximum Principle and of  Hamilton-Jacobi-Bellman equations (see e.g. \cite{SV},  \cite{MiRu}, \cite{KDPS} and  \cite{MR1}, \cite{MS}). In  order to have a representation formula for limit solutions associated to  controls with unbounded variation, in  \cite{MSuv} we singled out the following set  of controls:
$$
\overline{BV}_{loc}(T):= \{ u:[0,T]\to U, \ \ u\in {BV}_{loc}(T)\},
$$
for which we extended the graph completion approach. Precisely,  in  \cite{MSuv} we  introduced graph completions solutions associated to these controls and proved that they coincide  with the following subset of S limit solutions.   

\begin{definition}{\sc (BV$_{loc}$S limit solution)}\label{lsl}
Let   $(\bar x_0,\bar u_0) \in {\mathbb{R}}^n\times U$ and let  $(u,v)\in \overline{BV}_{loc}(T)\times L^1(T)$ with $u(0)=\bar u_0$.
An  S  limit solution  $x$ is called a    BV$_{loc}$S {\em limit  solution} of  (\ref{E})-(\ref{EO}):  
\begin{itemize}
\item[{(i)}]  {\em on $[0,T[$},  if there exists a sequence of controls  $(u_k)$ as in the  definition of  S   limit solution,   such that
 for any $t\in]0,T[$  the approximating inputs  $u_k$ have equibounded variation on $[0,t]$; 
  {\vskip 0.2 truecm}
 \item[{(ii)}]  {\em  on $[0,T]$},  if, moreover, $x$   is bounded and  there exists  a  decreasing map   $\tilde\varepsilon$ with $\lim_{s\to+\infty}\tilde\varepsilon(s)=0$ and  there exist  two strictly increasing, diverging sequences $(\tilde s_j)\subset{\mathbb{R}}_+$,   $(k_j)\subset{\mathbb{N}}$, $k_j\ge j$,  such that,  for every $k>k_j$ there is  $\tau^j_k <T$ with $ \tau^j_k+ Var_{[0,\tau^j_k]}(u_k)=\tilde s_j$ and
\begin{equation}\label{cBVl}
 |(x_k,u_k)(\tau^j_k)- (x_k,u_k)(T)|\le\tilde\varepsilon(j).  
\end{equation} 
\end{itemize}
\end{definition}
The subclass of BV$_{loc}$S limit solutions is  relevant  in controllability issues, like approaching a target set,  and  in  optimization problems with endpoint constraints and certain  running costs lacking coercivity  (see e.g.   Example 3.1 in \cite{MSuv}, involving the Brockett nonholonomic integrator).
 
  \begin{remark}\label{condST}{\rm  Condition (ii) in Definition \ref{lsl} is an {\it equiuniformity condition} on the sequence 
$(x_k,u_k)$ in a neighborhood of the final time $T$.  We point out that without  (\ref{cBVl}),  a BV$_{loc}$S limit  solution $x$ is a BV$_{loc}$ graph completion solution only on $[0,T[$.  Condition (ii)   guarantees the equivalence of the two concepts on the closed interval $[0,T]$ 
(see \cite{MSuv}).  }
\end{remark}
  
To better understand condition (ii) in Definition \ref{lsl}, for any traiectory-control pair $(x,u,v)$  let us introduce the following parametrization of the graph  of $(x,u)$,   useful also  in the sequel.
\begin{definition}[Arc-length parametrization]\label{Dal} 
Let $(u,v)\in  AC(T)\times L^1(T)$ with $u(0)=\bar u_0$   and  set $x:=x[\bar x_0,\bar u_0,u,v]$.   We call {\rm arc-length   graph-parametrization} of  the trajectory-control pair $(x,u,v)$,  the element  $(\xi,\varphi_0,\varphi,\psi, S)$ defined by  \footnote{ Since  every $L^1$ equivalence class contains  Borel measurable representatives, here and in the sequel we  tacitly assume that the maps $v$ and $\psi$ are Borel measurable,   when necessary.}
\begin{equation}\label{Rep}
\begin{array}{l}
 \sigma(t) {:=}  \int_0^t (1 +|\dot u(\tau)|) d\tau \quad\forall t\in[0,T], \  \ S:= \sigma(T)  \\ \, \\
\varphi_0{:=}  \sigma^{-1},\ \ \varphi{:=} u\circ \varphi_0,  \ \ \psi{:=}  v \circ \varphi_0 ,  \ \xi:= x\circ\varphi_0.
\end{array}
\end{equation}
Of course,  $(\xi,\varphi,\psi)\circ\sigma=(x,u,v)$.
\end{definition}

Notice that, given   $(\xi,\varphi_0,\varphi,\psi,S)$ defined as above, $(\varphi_0,\varphi)(0)=(0,\bar u_0)$, $\varphi_0(S)=T$ and   $\xi$ solves    the  following control system 
\begin{equation}\label{SEaff}
\left\{
\begin{array}{l}
\xi'(s)  =  g_0(\xi,\varphi,\psi) \varphi_0'(s)+ \sum_{i=1}^m {  g}_i(\xi,\varphi) {\varphi'_i}(s)  \quad s\in]0,S[, \\\,\\
 \xi(0) =\bar x_0. 
\end{array}
\right.
\end{equation}
Here the apex  `\,$'$\,'  denotes differentiation with respect to the parameter $s$, in order to distinguish it from the time differentiation, denoted by a dot.

\noindent Differently from the original solution $x$, which is  defined on the fixed time interval  $[0,T]$ and depends on an unbounded   control derivative $\dot u$,  the map $\xi$ is defined on a control-dependent interval  $[0,S]$ with $S=T+Var_{[0,T]}(u)\ge T$  but with $( \varphi_0',\varphi')$ bounded valued, since  $\varphi'_0+|\varphi'|=1$ a.e. in $[0,S]$.
  
\noindent Condition (ii) in Definition \ref{lsl}  is more meaningful once we read it as an hypothesis on  the graphs of the approximating sequence $(x_k,u_k)_k$. Precisely,  for any trajectory-control pair $(x_k,u_k,v)$ as in  Definition \ref{lsl}, let $(\xi, \varphi_{0_k},\varphi_k, v\circ \varphi_{0_k}, S_k)$ be its arc-length graph parametrization (see Definition \ref{Dal}).  Then    (ii)   is equivalent to: 

\noindent {\it the existence of   a positive, decreasing map  $\tilde\varepsilon$ with $\lim_{s\to+\infty}\tilde\varepsilon(s)=0$ and  of two strictly increasing, diverging sequences $(\tilde s_j)\subset {\mathbb{R}}_+$  and  $(k_j)\subset {\mathbb{N}}$, $k_j\ge j$,  such that,  for every $k>k_j$: 
\begin{equation}\label{cBVe}
|(\xi_k,\varphi_k)(\tilde s_j)- (\xi_k,\varphi_k)(S_k)|\le\tilde\varepsilon(j).
\end{equation}}
Clearly, (\ref{cBVe}) holds true when the sequence $(\xi_k,\varphi_k)$ is uniformly convergent on $ {\mathbb{R}}_+$ (by considering, for every $k$,  the  extension  $(\xi_k,\varphi_k)(s):=(\xi_k,\varphi_k)(S_k)$ for every $s\ge S_k$).


\section{Extended limit solution}\label{mr}
Motivated by Examples \ref{Ex3}, \ref{Ex3bis},  we extend here the  notions  of  limit solution  given in \cite{AR},  \cite{MSuv}, by approximating in $L^1$ the ordinary control $v$, which in  the original definitions  was kept fixed.  Furthermore, in Theorem \ref{Pvk} we prove that extended and original $BVS$  and $BV_{loc}S$ limit solutions, respectively, coincide.  Hence the results in \cite{AR}, \cite{AMR} and in \cite{MSuv}, dealing with $BVS$  and $BV_{loc}S$ limit solutions, remain unchanged in the new extended framework.

\begin{definition}[{\sc Extended limit solutions}]\label{edsdef}   
Let   $(\bar x_0,\bar u_0) \in {\mathbb{R}}^n\times U$ and let  $(u,v)\in {\mathcal L}^1(T)\times L^1(T)$ with $u(0)=\bar u_0$.
\begin{enumerate}
\item {\sc (E-Limit Solution)} A map $x\in{\mathcal L}^1([0,T], {\mathbb{R}}^n)$ is called an {\em extended limit solution,}  shortly E-{\em limit solution,} of the Cauchy problem (\ref{E})-(\ref{EO}) corresponding to  $(u,v)$
 if, {\em for every $\tau \in [0,T],$}  there is a sequence of controls $(u^\tau_k, v^\tau_k)\subset  AC(T)\times L^1(T)$   such that $u^\tau_k(0)=\bar u_0$ and
  {\vskip 0.2 truecm}
\begin{itemize}
\item[{(i$_\tau$)}] the sequence $(x^\tau_k)$ of  the  Carath\'eodory solutions $x^\tau_k:=  x[\bar x_0,\bar u_0, u^\tau_k,v^\tau_k]$ to (\ref{E})-(\ref{EO}) is equibounded on $[0,T]$;
 {\vskip 0.2 truecm}
\item[{(ii$_\tau$)}]
$|(x^\tau_k,u^\tau_k)(\tau)- (x,u)(\tau)|+ \|(x^\tau_k,u^\tau_k,v^\tau_k)- (x,u,v)\|_{_{L^1(T)}}\to 0$ as $k\to+\infty$.
\end{itemize}
 {\vskip 0.2 truecm}
\item {\sc (E-S limit solution)}
A limit solution  $x$ is called an  E-{\em simple limit solution} of (\ref{E})-(\ref{EO}), shortly  E-S   {\em limit solution},  if the sequences $(u^\tau_k, v^\tau_k)$ can be chosen independently of $\tau.$ In this case we write $(u_k, v_k)$ to refer to the approximating sequence.
\item {\sc (E-BVS  limit solution)}\label{sedsdef}
An E-S  limit solution  $x$ is  called an   E-BVS  {\em  limit  solution}, of (\ref{E})-(\ref{EO})   if   the approximating inputs  $u_k$ have equibounded variation on $[0,T]$.
\end{enumerate}
\end{definition}
 
\begin{definition}[{\sc  Extended BV$_{loc}$S limit solution}]\label{edsdefnew} 
Let  $(\bar x_0,\bar u_0) \in {\mathbb{R}}^n\times U$ and let  $(u,v)\in  \overline{BV}_{loc}(T))\times L^1(T)$ with $u(0)=\bar u_0$.
  An E-S  limit solution  $x$ is called an {\em extended}  BV$_{loc}$S {\em limit  solution,}  shortly  E-BV$_{loc}$S {\em limit  solution,}  of  (\ref{E})-(\ref{EO}):  
\begin{itemize}
\item[{(i)}]  {\em  on $[0,T[$},  if there exist a sequence of controls  $(u_k,v_k)$ as in the  definition of an E-S   limit solution,   such that
 for any $t\in]0,T[$  the approximating inputs  $u_k$ have equibounded variation on $[0,t]$; 
  {\vskip 0.2 truecm}
 \item[{(ii)}]  {\em  on $[0,T]$},  if, moreover, $x$   is bounded and  there exists  a  decreasing map   $\tilde\varepsilon$ with $\lim_{s\to+\infty}\tilde\varepsilon(s)=0$ and  there exist  two strictly increasing, diverging sequences $(\tilde s_j)\subset{\mathbb{R}}_+$,   $(k_j)\subset{\mathbb{N}}$, $k_j\ge j$,  such that,  for every $k>k_j$ there is  $\tau^j_k <T$ with $ \tau^j_k+ Var_{[0,\tau^j_k]}(u_k)=\tilde s_j$ and
\begin{equation}\label{cBVl1}
 |(x_k,u_k)(\tau^j_k)- (x_k,u_k)(T)|\le\tilde\varepsilon(j).  
\end{equation} 
\end{itemize}
\end{definition}

Analogously to the case of limit solutions, the extended limit solution associated to a control   $(u,v)\in {\mathcal L}^1(T)\times L^1(T)$ is not unique, unless  the system is commutative; moreover  the sets of E-limit solutions, E-S, E-BV$_{loc}$S, and E-BVS   limit solutions are a decreasing sequence of sets.

\begin{theorem}\label{Pvk}  Let $T>0$, $(\bar x_0,\bar u_0)\in {\mathbb{R}}^n\times U$ and let  $(u,v)\in {\mathcal L}^1(T)\times L^1(T)$ be such that $u(0)=\bar u_0$.  Then 
  a map   $x:[0,T]\to{\mathbb{R}}^n$ is an   E-BVS limit solution [resp.  E-BV$_{loc}$S limit solution]   corresponding to  $(u,v)$ if and only if it is a  BVS  limit solution [resp. BV$_{loc}$S limit solution]  corresponding to the same input.
  \end{theorem}
\begin{proof} \label{Confr}

The `` if'' part is obvious for both cases. Let us prove the ``only   if'' part.
 
{\vskip 0.2 truecm} 
\noindent{\sc \underline{Case 1}:} Let    $x$ be  an E-BVS limit solution corresponding to $(u,v)$ and let $(u_k, v_{k})$ and $(x_k)$ be as in Definition \ref{edsdef}, so that, in particular,  there is some constant  $K>0$  such that $Var_{[0,T]}(u_k)\le K$ for every $k$. Then, setting $\hat x_k:=x[\bar x_0,\bar u_0,u_k,v]$, by standard estimates  it follows that
\begin{equation}\label{eqv_k}
|x_k(t)|, \  |\hat x_k(t)|   \le R'
\end{equation}
with $R':=[|\bar x_0|+(m+1)M(T+K)]{\rm e}^{(m+1)M (T+K)}.$
 Let us denote by $\omega$ and $L$ a modulus of continuity of $g_0$ and a Lipschitz constant (in $(x,u)$) for the vector fields $g_i$, $i=0,\dots m$ when $|x|\le  R'$, respectively.  Gronwall's Lemma yields that
\begin{equation}\label{eqv_k2}\begin{array}{l}
 |\hat x_k(t)-x_k(t)|\le \\\left(\int_0^t \omega(|v_k(t')-v(t')|) \,dt'\right) \,\,e^{(m+1)L\left(t+\int_0^t|\dot u_k(t')|\,dt'\right)}. 
\end{array}\end{equation}
Since there exists a subsequence of $(v_k)$ such that $v_k(t)\to v(t)$ a.e. in $[0,T]$ and $v$, $v_k$ take values in the compact set $V$,  the 
Dominated Convergence Theorem and the continuity of $\omega$ let  us conclude that,  for such a subsequence, 
\begin{equation}\label{omega1}
\int_0^T\omega(|v_k(t)-v(t)|)\,dt\to 0,\quad{\rm as} \,\, k\to+\infty
\end{equation}
so that 
$\lim_k|\hat  x_{k}(t)-x_{k}(t)|=0$  for every  $t\in [0,T]$. Therefore, $\lim_k \hat x_k(t)=\lim_k x_k(t)=x(t)$  for any  $t\in[0,T]$ and $x$ is a $BVS$  limit solution corresponding to $(u,v).$

{\vskip 0.2 truecm} 
\noindent{\sc \underline{Case 2}:} Let  now $x$ be  an  E-BV$_{loc}$S, not  E-BVS, limit solution  and let $(u_k, v_{k})$,  $(x_k)$, $(k_j)$  and $(\tilde s_j)$ be as in Definition \ref{edsdefnew}.  For every $k$, set $V_k: =Var_{[0,T]}(u_k)$ and assume that  $(V_k)$  is increasing and diverging. By (i) in Definition \ref{edsdefnew}  there exists an increasing function  $V:[0,T[\to{\mathbb{R}}_+$   with $V(0)=0$, $\lim_{t\to T}V(t)=+\infty$ and  such that, for every $k$, 
$$
Var_{[0,t]}(u_k)\le V(t) \qquad \rm{for\, every }\   t\in]0,T[.
$$
Then by the proof of Case 1 we derive that  
$$
\hat x_k(t):=x[\bar x_0,\bar u_0,u_k,v](t)\to x(t) \quad\rm{ for \,every }\  t\in[0,T[.
$$ 
To handle the convergence at  $t=T$, we use part  (ii) of the definition of  E-BV$_{loc}$S limit solution.
  Let us  introduce, for every $k$,  the arc-length graph parametrizations $(\xi_k,\varphi_{0_k}, \varphi_k,v_k\circ\varphi_{0_k},T+V_k)$ and $(\hat\xi_k,\varphi_{0_k}, \varphi_k,v\circ\varphi_{0_k},T+V_k)$ of $(x_k, u_k,v_k)$ and $(\hat x_k, u_k,v)$, respectively (see Definition \ref{Dal}).  Let us suppose that these  arc-length graph parametrizations are extended to $[T+V_k,+\infty[$  by the  constant value assumed at $T+V_k$. 
By assumption,  there exists a constant $R>0$ such that  
$$
\sup_{s\in{\mathbb{R}}_+}|\xi_k(s)|=\sup_{t\in[0,T]}|x_k(t)|\le R \qquad \rm{for \,every }\  {\it k}
$$
and, recalling that  $\varphi'_{0_k}(s)+|\varphi_k'(s)|\le1$  a.e.,  standard estimates imply that  for any $j$   there is some  $R_j>0$ such that 
$$
\sup_{s\in[0,\tilde s_j]}|\hat\xi_k(s)| \le R_j \qquad \rm{for \,every} \   {\it k}.
$$
Let $\omega_{j}$ and $L_j$ be a modulus of continuity of $g_0$ and a Lipschitz constant (in $(x,u)$)  of the vector fields $g_i,\, i=0,\dots,m$  for $|x|\le \max\{R,R_j\}$, respectively.   Gronwall's Lemma yields,  for every $k$, 
  \begin{equation}\label{x4}\begin{array}{l}
\sup_{t\in[0,\tau^j_k]}|\hat x_k(t)-x_k(t)|= \sup_{[0,\tilde s_j]} |\hat\xi_k(s)-\xi_k(s)|\le\\\,\\
\int_0^{\tilde s_j}\omega_{j}(|(v_k-v)\circ\varphi_{0_k}(r)|)\varphi'_{0_k}(r)\,dr\,\cdot\\
\quad\quad\quad\quad\quad\quad e^{(m+1)L_j\int_0^{\tilde s_j}
 (\varphi'_{0_k}(r)+|\varphi'_k(r)|)\,dr}\le\\\,\\
\int_0^T\omega_{j}(|v_k(t)-v(t)|)\,dt\,\,e^{(m+1)L_j\tilde s_j}=:\varepsilon^2_j(k)
\end{array}
\end{equation}
 with $\varepsilon^2_j(k)\le \varepsilon^2_{j+1}(k)$. Passing to a suitable subsequence of $(v_k)$, still denoted by $(v_k)$, as in (\ref{omega1})  we  have that,  for every fixed  $j$,  $\lim_k\varepsilon^2_j(k)=0$.  Now we can construct a sequence $(k^1_j)$, with $k^1_j\ge k_j$, such that  
\begin{equation}\label{x4}\varepsilon^2_j(k)\le1/j\quad{\rm for\,\,all\,\,}k\ge k^1_j.\end{equation}
 
  In particular,   this implies that, for some $\hat R>0$,    
$$
\sup_{[0,\tau^j_k]}|\hat x_{k}|\le  \hat R  \qquad \forall k\ge k^1_j.
$$
Since $\lim_k V_k=+\infty$, we need to modify the  sequence $(\hat x_k,\hat u_k)$ using the Whitney property.  Precisely, we set $\tau^j:=\tau^j_{k^1_j}$  and 
\begin{equation}\label{Wu}
\begin{array}{l}
\check u_j:=\hat u_{k^1_j}(t)\chi_{[0,\tau^j[}(t)+{\tilde u}_j\left(\frac{t-\tau^j}{T-\tau^j}\right)\chi_{]\tau^j,T]},  \\ \, \\
\check x_j:=x[\bar x_0,\bar u_0,\check u_j,v],
\end{array}
\end{equation}
 where  ${\tilde u}_j\in AC(1)$ joins  $\hat u_{k^1_j}(\tau^j)=\varphi_j(s_j)$ to  $u(T)$ and $Var_{[0,1]}{\tilde u}_j\le C| \varphi(s_j)- u (T)|$.
 We have $\check x_{j}(\tau^j)=\hat x_{k^1_j}(\tau^j),$ and by standard estimates it follows that
 $\sup_{t\in[0,T]}|\check  x_j(t)|\le \check  R$  for some $\check R>0,$ and   
\begin{equation}\label{hats}
|\check  x_j(T)- \check x_{j}(\tau^j)|\to 0\quad{\rm as}\,\,j \to+\infty.
\end{equation}
Hence by (\ref{x4}), {\rm(\ref{cBVl})}  and (\ref{x4}) we get
\begin{equation}\label{iii}\begin{array}{l}
|\check  x_j(T)- x(T)|\le |\check  x_j(T)- \check  x_{j}(\tau^j)|+|\hat  x_{k^1_j}(\tau^j)- x_{k^1_j}(\tau^j)|+\\\,\\|x_{k^1_j}(\tau^j)-x_{k^1_j}(T)|+ |x_{k^1_j}(T)- x(T)|\le\\\,\\
|\check  x_j(T)- \check  x_{j}(\tau^j)|+\frac{1}{j}+\tilde\varepsilon(j) + |x_{k^1_j}(T)- x(T)|.
  \end{array}
\end{equation}
The r.h.s. of (\ref{iii}) approaches $0$ since  by (\ref{hats})   its first term goes to $0$ and, being
$x$ an E-BV$_{loc}$S limit solution,  the  last term  approaches $0$ too. Therefore, renaming  the index $j$ in the sequence $(\check x_j,\check u_j)$ by $k$, it is not difficult to prove that the sequence  $(\check x_k,\check u_k)$   verifies statements (i) and (ii)  and, by (\ref{iii}),  also (ii) of Definition \ref{edsdef}. 
\end{proof}  


\section{A further extension}\label{Sfex}
For $u$  with bounded variation, the graph completion technique has been extended  since the 90s   to control systems of the form
\begin{equation}\label{Ev}
 \dot x (t)= g_0(x(t),u(t),v(t)) +  \sum_{i=1}^m  {g}_i(x(t),u(t),v(t))\,\dot u_i(t) \quad \text{a.e. $t\in[0,T]$,} 
 \end{equation}
\begin{equation}\label{EOv} 
x(0)=\bar x_0,  \quad u(0)=\bar u_0,
\end{equation}
 where the dependence on the ordinary control $v$ appears also in the coefficients $g_1, \dots, g_m$ of the control derivatives $\dot u_i$. This notion   
 has been applied to several problems  (see  \cite{MR},   \cite{MiRu},  \cite{KDPS}  and the references therein). As mentioned in \cite{AR},   this  kind of equation  is  relevant  in mechanical applications, for instance,  when $u$ is a shape parameter and $v$ is a control representing an external force or torque  and in min-max control problems where  the adjoint equations may contain a $v$-dependent term multiplied by an unbounded control, like in \eqref{Ev} (see e.g. \cite{BI}). In this section we  adapt the notion of extended  BVS  limit solution introduced in Definition \ref{edsdef} to   \eqref{Ev},  \eqref{EOv}  and  in Theorem \ref{Th2} below we  prove the one-to-one correspondence  between such  limit solutions and  graph completion solutions to \eqref{Ev},  \eqref{EOv}.  In this way we extend  the result  of \cite[Thm. 4.2]{AR}, where the same assertion  is proved for  $g_1,\dots,g_m$ independent of $v$.
 {\vskip 0.2 truecm}
 Throughout this section we assume that for every $i=0,\dots,m$,  the control vector field  $g_i: {\mathbb{R}}^n \times U \times V\to{\mathbb{R}}^n$  is continuous, $ (x,u) \mapsto  g_i(x,u,v)$
 is locally Lipschitz on ${\mathbb{R}}^n \times U$ uniformly in $v \in V$ and there exists $M>0$ such that
  $$
  \left|g_i(x,u,v)\right| \leq  M(1+|(x,u)|) \quad \forall  (x,u,v) \in {\mathbb{R}}^n\times U\times V.
  $$
The notion of  extended BVS limit solution to  \eqref{Ev},  \eqref{EOv} that we are going to introduce  coincides with the Definition \ref{edsdef}, 3.,   for   $g_1,\dots,g_m$ not depending on $v$, but the presence of   the ordinary control in the $g_i$ for $i=1,\dots,m$  requires to take into account the interplay between $u$ and $v$. We distinguish the two situations ($v$ just in the drift or $v$ `everywhere') by considering  the more general control system 
\begin{equation}\label{Evg}
 \dot x (t)= g_0(x(t),u(t),v_1(t)) +  \sum_{i=1}^m  {g}_i(x(t),u(t),v_2(t))\,\dot u_i(t) \quad \text{a.e. $t\in[0,T]$,} 
 \end{equation}
with $v:=(v_1,v_2)$ taking values in $V\times V$. For simplicity, we use the same notation of  Definition \ref{edsdef} and still  denote by $x[\bar x_0,\bar u_0,u,v]$ a regular solution to \eqref{Evg}, \eqref{EOv} associated to $(u,v)=(u,v_1,v_2)\in AC(T)\times L^1(T)\times L^1(T)$. 

\begin{definition}[{\sc Extended BVS limit solution}]\label{edsdefv}   
Let   $(\bar x_0,\bar u_0) \in {\mathbb{R}}^n\times U$ and let  $(u,v)=(u,v_1,v_2)\in {\mathcal L}^1(T)\times L^1(T)\times L^1(T)$ with $u(0)=\bar u_0$.
\begin{enumerate}[]
\item  A map $x\in{\mathcal L}^1([0,T], {\mathbb{R}}^n)$ is called an {\em extended} BVS {\em limit solution,}  shortly E-BVS {\em limit solution,} of the Cauchy problem \eqref{Evg}-\eqref{EOv} corresponding to  $(u,v)$
 if  there is a sequence of controls $(u_k, v_{k})=(u_k, v_{1_k}, v_{2_k})\subset  AC(T)\times L^1(T)\times L^1(T)$   such that $u_k(0)=\bar u_0$, the approximating inputs  $u_k$ have equibounded variation on $[0,T]$ and 
  {\vskip 0.2 truecm}
\begin{itemize}
\item[(i)] the sequence $(x_k)$ of  the  Carath\'eodory solutions $x_k:=  x[\bar x_0,\bar u_0, u_k,v_k]$ to \eqref{Evg}-\eqref{EOv} verifies for every $\tau\in[0,T]$, 
 {\vskip 0.2 truecm}
$|(x_k,u_k)(\tau)- (x,u)(\tau)|+ \|(x_k,u_k,v_k)- (x,u,v)\|_{_{L^1(T)}}\to 0$ as $k\to+\infty$;
{\vskip 0.2 truecm}
\item[{(ii)}] there is some $\psi_2\in L^1({\mathbb{R}}_+, V)$ such that, setting $\sigma_k(t):= t+Var_{[0,t]}(u_k)$,   $V_k:=$Var$_{[0,T]}(u_k)$, one has  
$\|(v_{2_k}\circ(\sigma_k)^{-1}-\psi_2)\,\chi_{_{[0,T+V_k]}}\|_{_{L^1({\mathbb{R}}_+)}}\to 0$  as $k\to+\infty$.
\end{itemize}
  \end{enumerate}
\end{definition}  
 
 \begin{theorem}\label{Th2} A map $x:[0,T]\to{\mathbb{R}}^n$ is a  {\rm E-BVS}-limit solution to  \eqref{Evg}-\eqref{EOv} associated to $(u,v)\in BV(T)\times L^1(T)\times L^1(T)$ with $u(0)=\bar u_0$ if and only if it is a graph completion solution to  \eqref{Evg},  \eqref{EOv} associated to the same control. 
\end{theorem} 

Before proving the theorem, let us  briefly describe the graph completion approach and  give the precise definition of  graph completion solution to   \eqref{Evg}, \eqref{EO}. For more details we refer  the interested reader   to \cite{MR} and the references therein.
 {\vskip 0.2 truecm}
\noindent For $L>0$ and $S>0$,    let ${\mathcal U}_L(S)$  denote the subset of $L$-Lipschitz  maps 
 $$(\varphi_0,\varphi):[0,S]\to {\mathbb{R}}_+\times U,$$ 
 such  that $\varphi_0(0)=0$, and $\varphi_0'(s) \geq 0$, $\varphi_0'(s)+|\varphi' (s)|\le L$ 
for   almost every  $s\in[0,S]$.  We set $L^1(S):=L^1([0,S], V)$.
 
We   call  {\em space-time controls} the elements $(\varphi_0,\varphi,\psi, S)=$ $(\varphi_0,\varphi,\psi_1,\psi_2, S)$  with $S>0$ and   $(\varphi_0,\varphi,\psi_1,\psi_2)\in \bigcup_{L>0}{\mathcal U}_L(S)\times L^1(S)\times L^1(S)$.   Let  $(\bar x_0,\bar u_0)\in {\mathbb{R}}^n\times U$. We denote by $\Gamma(\bar u_0)$ the subset of space-time controls  verifying   $(\varphi_0,\varphi)(0)=(0,\bar u_0)$ and $\varphi_0(S)=T$.  The {\rm space-time control system}  is defined by
\begin{equation}\label{SEv}
\left\{
\begin{array}{l}
\xi'(s)  =  g_0(\xi,\varphi,\psi_1) \varphi_0'(s)+ \sum_{i=1}^m {  g}_i(\xi,\varphi,\psi_2) {\varphi'_i}(s)  \quad \text{for a.e. $s\in[0,S]$,} \\ [1.5ex]
\xi(0)=\bar x_0
\end{array}
\right.
\end{equation}
 and we use  $\xi [\bar x_0,\bar u_0,\varphi_0,\varphi,\psi]$ to denote  its solution.  Notice that by just identifying regular controls $u$ and trajectories $x$ with their graphs and considering a time parametrization $t=\varphi_0(s)$,  \eqref{Ev} can be  embedded  in the space-time system \eqref{SEv}.  However, when a space-time control has  $t=\varphi_0(s)=const$ for $s\in I:=[s_1,s_2]$, the pair $(\xi,\varphi)$  describes on $I$ the `instantaneous evolution' at time $t$ of the system; this is a way to  define generalized controls and trajectories for the original control system in the extended, space-time setting.  Now any space-time trajectory-control pair  gives rise  to a {\it set-valued} notion of generalized solution $x(t):=\xi\circ\varphi_0^{-1}(t)$ to \eqref{Ev}, associated to a control $(u,v)$ with  $(u,v)(t)\in (\varphi,\psi)\circ\varphi_0^{-1}(t)$;  following \cite{AR}, a (univalued) concept of graph completion solution is then obtained  by the choice of a suitable selection.

\noindent Since the space-time control system \eqref{SEv} is  rate-independent, without loss of generality we consider just controls verifying
 $$
\varphi_0'(s)+|\varphi' (s)|=1\quad \text{for  a.e. $s\in[0,S]$. }
$$
 ${\Gamma}_f(\bar u_0)$ will  denote the subset of such controls, to which we  will refer to as  {\it feasible space-time controls}. 
   
\begin{definition}\label{GCloc}  Let  $(u,v)=(u,v_1,v_2) \in BV(T) \times L^1(T)\times L^1(T)$  and $u(0)=\bar u_0\in U$.  We say that a   space-time control     $(\varphi_0,\varphi,\psi, S)\in{\Gamma}_f(\bar u_0)$   is a {\em graph completion   of $(u,v)$}  if   
$$
\forall t\in[0,T], \ \ \exists    s\in[0,S] \ \text{such that} \ 
 (\varphi_0,\varphi,\psi)(s)=(t,u(t),v(t)).
 $$
 We call a {\it clock} any strictly  increasing, surjective  function $\sigma:[0,T]\to[0,S]$ such that 
$$
(\varphi_0,\varphi)(\sigma(t))=(t,u(t)) \ \text{ for every $t\in[0,T]$.}  
$$
\end{definition}

 \begin{definition}\label{GClocS}    Given a control $(u,v) \in BV(T) \times L^1(T)\times L^1(T)$  with $u(0)=\bar u_0$, let     
 $(\varphi_0,\varphi,\psi, S)$ be a  graph- completion  of $(u,v)$ and   let $\sigma$ be a clock. Set   $\xi:=\xi[\bar x_0, \bar u_0,\varphi_0,\varphi,\psi]$.  
 A map 
$$
x:[0,T]\to{\mathbb{R}}^n, \quad x(t){:=} \xi\circ\sigma(t)\  \ \forall t\in[0,T],
$$
is called a {\it graph completion solution} to \eqref{Evg}, \eqref{EO}.
\end{definition}
  
\begin{proof}[Proof of Theorem  \ref{Th2}]  Let  $(u,v) \in BV(T) \times L^1(T) \times L^1(T)$  and $u(0)=\bar u_0\in U$.   We  begin by showing that a     graph completion solution $x$  to \eqref{Evg}, \eqref{EOv} associated to $(u,v)$  is a E-BVS limit solution.
By Definitions \ref{GCloc} and \ref{GClocS},  there exist a feasible  space-time control   $(\varphi_0,\varphi,\psi, S)\in {\Gamma}(\bar u_0)$  and a surjective, strictly increasing function \linebreak $\sigma:[0,T]\to[0,S]$ such that, setting  $\xi:=\xi[\bar x_0,\bar u_0, \varphi_0,\varphi,\psi]$, one has 
\begin{equation}\label{fsj}
(\xi,\varphi_0,\varphi,\psi)\circ\sigma(t)=(x(t),t, u(t),v(t))   \qquad \forall t\in[0,T] .
\end{equation}
By \cite[Thm. 5.1]{AR} as revisited in  \cite[Thm. 4.2]{MSuv}, there exists a sequence $(\sigma_k)$ of absolutely continuous, strictly increasing  maps $\sigma_k:[0,T]\to[0,S]$, such that 
\begin{itemize}
\item[(i)]  $\sigma_k(0)=0$,   $\sigma_k(T)=S$,  and
\begin{equation}\label{cp}
 \displaystyle \dot\sigma_k(t)\ge 1 \ \ \text{ for a.e. $t\in[0,T]$,}  \quad \lim_{k\to+\infty}\sigma_k(t)=\sigma(t) \quad \forall t\in[0,T];
\end{equation}
\item[(ii)]  the maps $\varphi_{0_k}:=\sigma_k^{-1}:[0,S]\to[0,T]$ are strictly increasing, $1$-Lipschitz continuous,  surjective and  converge uniformly to $\varphi_0$ in $[0,S]$. 
\end{itemize}
We are going to show that  the sequences $(u_k,v_k)$ and $(x_k)$  defined by
$$
u_k:=\varphi\circ\sigma_k,\,\, v_k:=\psi\circ\sigma_k,\,\,   x_k:=x[\bar x_0,\bar u_0,u_k,v_k], 
$$
verify all the requirements of Definition \ref{edsdefv}, so proving that $x$ is a E-BVS limit solution of  \eqref{Ev}, \eqref{EOv} associated to $(u,v)$.

\noindent In view of definition \eqref{fsj},  the pointwise  convergence of $u_k$ to $u$ follows from the continuity of $\varphi$. Moreover, the sequence $(u_k)$ has equibounded variation, since   Var$_{[0,T]}(u_k)=$Var$_{[0,S]}(\varphi)$ for every $k$.  In order to show that 
$ \displaystyle\lim_{k\to+\infty}\|v_k-v\|_{_{L^1(T)}}=0$,  take an arbitrary $\varepsilon>0$ and consider a bounded, continuous  map $\tilde\psi:[0,S]\to{\mathbb{R}}^{2l}$  such that 
$$
\int_0^S|\tilde\psi(s)-\psi(s)|\,ds<\varepsilon,
$$
(such $\tilde\psi$ exists  by well known density results).  Hence 
$$
\begin{array}{l}
\int_0^T|v_k(t)-v(t)|\,dt=\int_0^T|\psi(\sigma_k(t))-\psi(\sigma(t))|\,dt\le \int_0^T|\psi(\sigma_k(t))-\tilde \psi(\sigma_k(t))|\,dt+ \\ [1.5ex]
 \qquad\int_0^T|\tilde\psi(\sigma_k(t))-\tilde\psi(\sigma(t))|\,dt+\int_0^T|\tilde\psi(\sigma(t))-\psi(\sigma(t))|\,dt \le  \\ [1.5ex]
 \qquad\quad \int_0^T|\psi(\sigma_k(t))-\tilde \psi(\sigma_k(t))|\dot\sigma_k(t)\,dt+
\int_0^T|\tilde\psi(\sigma_k(t))-\tilde\psi(\sigma(t))|\,dt+ \\ [1.5ex]
 \qquad\quad \quad \qquad\quad \int_0^T|\tilde\psi(\sigma(t))-\psi(\sigma(t))|\,d\sigma(t), 
\end{array}
$$
where the last inequality follows from the properties of $\sigma$ and    $\sigma_k$.  Now the first and the third integrals in the r.h.s.,  by the  (continuous) change of variable   $s=\sigma_k(t)$ and the discontinuous one   $s=\sigma(t)$ (see e.g. \cite{FT})  respectively,  are both less than $\varepsilon$, while the second integral tends to $0$ by the Dominated Convergence Theorem, since   $\tilde\psi$ is bounded and continuous.  By the arbitrariness of $\varepsilon>0$, this    concludes  the proof that $ \displaystyle\lim_{k\to+\infty}\|v_k-v\|_{_{L^1(T)}}=0$. Since 
$$
v_{2_k}\circ\sigma_k^{-1}=\psi_2\circ\sigma_k\circ\sigma_k^{-1}\equiv \psi_2,
$$
 the condition $\|(v_{2_k}\circ\sigma_k^{-1}-\psi_2)\,\chi_{_{[0,T+V]}}\|_{_{L^1({\mathbb{R}}_+)}}\to 0$  as $k\to+\infty$ is trivially satisfied.
  
\noindent It remains to show that $x$ is the pointwise limit of $(x_k)$.  To this aim, let us set $\xi_k:= \xi[\bar x_0,\bar u_0,\varphi_{0_k}, \varphi,\psi]$.  By the continuity of the input-output map  associated to the control system \eqref{Ev} (see  \cite[Thm. 4.1]{MR}) we derive that   $(\xi_k)$ converges uniformly to $\xi$ on $[0,S]$. Since   $x_k=\xi_k\circ\sigma_k$  on $[0,T]$, we finally obtain  that, for every $t\in[0,T]$,  one has
$$
\displaystyle\lim_{k\to+\infty}|x_k(t)-x(t)|=\lim_{k\to+\infty}|\xi_k(\sigma_k(s))-\xi(\sigma(t))|=0. 
$$
Hence $x$ is a E-BVS limit solution.

 {\vskip 0.2 truecm}
 
Let us now show that an E-BVS limit solution $x$   to \eqref{Evg}, \eqref{EOv} associated to $(u,v)$ is a graph completion solution.  By Definition  \ref{edsdefv}, there exist $\psi_2\in L^1(T)$ and  a sequence  $(u_k,v_k)\subset  AC(T)\times L^1(T)\times L^1(T)$  with $u_k(0)=\bar u_0$ and $V_k:=$Var$(u_k)\le K$ for some $K>0$ such that,  setting 
\begin{equation}\label{sk}
\sigma_k(t):= t+Var_{[0,t]}(u_k)  \quad(\le S:=T+K)   
\end{equation}
and $x_k:=  x[\bar x_0,\bar u_0, u_k,v_k]$, one has
\begin{equation}\label{psi2}
\begin{array}{l}
\displaystyle\lim_{k\to+\infty}(x_k(t),u_k(t))=(x(t),u(t)) \quad \text{for any $t\in[0,T]$,} \\ [1.5ex]
\displaystyle\lim_{k\to+\infty}\int_0^T|v_k(t)-v(t)|\,dt=0,  \\ [1.5ex]
\displaystyle\lim_{k\to+\infty}\int_{{\mathbb{R}}_+}|v_{2_k}\circ\sigma_k^{-1}(s)-\psi_2(s)|\,\chi_{_{[0,T+V_k]}}\,ds=0
\end{array}
\end{equation} 
Arguing as   in the proof of Theorem \ref{Pvk}, Case 1, one can prove  that it is possible  to assume,  without loss of generality, that  $v_{1_k}=v_1$ for every $k$.  Let $\varphi_{0_k}:[0,S]\to[0,T]$ be   the $1$-Lipschitz continuous,  increasing function such that  
$$
\varphi_{0_k}:=\sigma_k^{-1} \ \text{ on $[0,T+V_k]$, and }  \ \varphi_{0_k}(s)=T \ \text{ for all $s\in]T+V_k, S]$.}
$$
Set $\varphi_k:= u_k\circ\varphi_{0_k}.$ Then the sequence of space-time controls  $(\varphi_{0_k},\varphi_k)$   is $1$-Lipschitz continuous on $[0,S]$ and satisfies $\varphi'_{0_k}(s)+|\varphi_k'(s)|=1$ for a.e. $s\in [0,T+V_k]$ (and  $\varphi'_{0_k}(s)+|\varphi_k'(s)|=0$ for $s>T+V_k$).  
 Therefore by Ascoli-Arzel\`a's Theorem, taking if necessary a subsequence,  still denoted by  $(\varphi_{0_k},\varphi_k)$, it converges  uniformly   to a  Lipschitz continuous function $(\varphi_0,\varphi)$ such that $\varphi'_{0}(s)+|\varphi'(s)|\le1$ for $s\in[0,S]$. 
Let us observe that  $(\varphi_0,\varphi)$ is a  graph completion of $u$, possibly not feasible (namely, not verifying the equality $\varphi'_{0}(s)+|\varphi'(s)|=1$ a.e.). Indeed, for every $t\in[0,T]$,  there exist a subsequence $(\sigma_{k'}(t))$ and $\sigma(t)\in [0,S]$  such that $\lim_{k'}\sigma_{k'}(t)=\sigma(t)$. Therefore, by the uniform convergence of  $(\varphi_{0_k},\varphi_k)$  it follows that 
$$
(\varphi_0,\varphi)\circ\sigma(t)=\displaystyle\lim_{k'\to+\infty}(\varphi_{0_{k'}},\varphi_{k'})\circ\sigma_{k'}(t)=(t,u(t)).
$$
 Set  
 $$
 \psi_1:=v\circ\varphi_0, \quad \psi:=(\psi_1,\psi_2),
 $$
 where $\psi_2$ is the same as in \eqref{psi2}   and define the solution $\xi:=\xi[\bar x_0,\bar u_0, \varphi_0,\varphi, \psi]$ associated to the space-time control $(\varphi_0,\varphi,\psi, S)$.  Moreover, let $\psi_k:=(v_1\circ\varphi_{0_k},  v_{2_k}\circ\varphi_{0_k})$ and $\xi_k:=\xi[\bar x_0,\bar u_0, \varphi_{0_k},\varphi_k, \psi_k]$. Clearly,  $x_k=\xi_k\circ \sigma_k$. In order to prove that  $x$ is a graph completion solution, let us first verify that $x=\xi\circ\sigma$.  To this aim, we observe that this is true as soon as there exists a subsequence of $(\xi_k)$ uniformly converging in $[0,S]$ to $\xi$. In this case indeed,  for every $t\in[0,T]$,   the pointwise convergence of $ \sigma_{k'}(t)$ to $\sigma(t)$ implies that
$$
x(t)=\lim_{k'}x_{k'}(t)=\lim_{k'} \xi_{k'}\circ \sigma_{k'}(t)=\xi\circ\sigma(t).
$$
At this point, if we introduce the change of variable
 $$
 \eta(s):=\int_0^s\left[\varphi_0'(r)+|\varphi'(r)|\right]\,dr \quad \forall s\in[0,S], \quad \tilde V:= \eta(S)-T,
 $$
denote by $s(\cdot): [0,T+\tilde V]\to[0,S]$ its  the strictly increasing  right-inverse, define the  feasible space-time control 
 $$
 (\tilde\varphi_0,\tilde\varphi, \tilde\psi,\tilde S):=(\varphi_0\circ s,\varphi\circ s,\psi\circ s, T+\tilde V),
 $$
 and the clock $\tilde\sigma:=\eta\circ\sigma$, we can easily obtain that $x$ is a graph completion solution, since   
 $$
 x=\xi\circ\sigma=\tilde\xi\circ\tilde\sigma \qquad (\tilde\xi:=\xi[\bar x_0,\bar u_0,\tilde\varphi_0,\tilde\varphi, \tilde\psi]).
  $$
 To conclude the proof it remains to show that, eventually for a subsequence, one has 
\begin{equation}\label{cunif}
\displaystyle\lim_{k\to+\infty}\sup_{s\in[0,S]}\,|\xi_k(s)-\xi(s)|=0.
\end{equation}
Since both the derivatives $(\varphi_{0_k}', \varphi')$, $(\varphi'_0,\varphi')$ are bounded, by standard estimates it follows that  
$$
\sup_{s\in[0,  S]}|\xi(s)|,\ \sup_{s\in[0, S]}|\xi_k(s)|\le \bar M:=(|\bar x_0|+(m+1)M  S)e^{(m+1)M S}.
$$
Let us denote by $\omega$ a modulus of continuity of $g_0(x,u,\cdot), \dots, g_m(x,u,\cdot)$,   by $\tilde L$    a   Lipschitz constant of  $g_0, \dots, g_m$ in $(x,u)$ uniformly w.r.t. $v$,   and by   $\tilde M$ an upper bund  for all the vector fields $g_i$, $i=0,\dots m$,   in the compact set $\overline{B_n(0,\bar M)}\times U\times V$.  After some calculations,  setting
$$
\begin{array}{l}
\displaystyle  f_k(s):=\bigg| \int_0^{s}\big[g_0(\xi(r),\varphi(r),v_1\circ\varphi_{0}(r))[\varphi'_{0_k}(r)-\varphi'_0(r)] + \\ [1.5ex]
\qquad\qquad\qquad\qquad\sum_{i=1}^m g_i(\xi(r),\varphi(r), \psi_2(r))[\varphi'_{i_k}(r)-\varphi'_i(r)]\big]\,dr\bigg|,
\end{array} 
 $$ 
 and
 $$
 \begin{array}{l}
 \rho_{1_k}:= \int_0^{S} \omega (|v_1\circ\varphi_{0_{k}}(r)-v_1\circ\varphi_{0}(r)|)\varphi'_{0_k}(r)\,dr,   \\ [1.5ex]
  \rho_{2_k}:= \sum_{i=1}^m\int_0^{S} \omega (|v_{2_k}\circ\varphi_{0_{k}}(r)-\psi_2(r)|)\,|\varphi'_{i_k}(r)|\,dr,
  \end{array}
 $$
by the Gronwall's Lemma  we get to  
\begin{equation}\label{4r}
\displaystyle|\xi_k(s)-\xi(s)|\le  \, e^{(m+1)\tilde LS}\, \left(\sup_{s\in[0,S]} f_k(s)+\rho_{1_k}+\rho_{2_k}\right).
\end{equation}
 The uniform convergence of $(\varphi_{0_k},\varphi_k)$ to $(\varphi_{0},\varphi)$ on $[0, S]$  implies that the maps $(\varphi'_{0_k},\varphi'_{k})$ tend to $(\varphi'_0, \varphi')$ in  the weak$^*$ topology   of $L^\infty([0, S], {\mathbb{R}}^{1+m})$, so that $f_k(s)$
 tends to 0 as $k\to+\infty$ for every $s\in[0,S]$.    The uniform convergence to 0 of the $f_k$'s now follows from Ascoli-Arzel\'a Theorem, for the $f_k$'s are equibounded and equi-Lipschitzean.  By \eqref{psi2}  and the  inequality $|\varphi'_{i_k}|\le 1$ a.e., we derive that $\displaystyle\lim_{k\to+\infty}\rho_{2_k}=0$.   By a time-change, we get
$$
\int_0^{S}|v_1\circ\varphi_{0_{k}}(r)-v_1\circ\varphi_{0}(r)|\varphi'_{0_k}(r)\,ds=\int_0^{T}|v_1(t)-v_1\circ\varphi_{0}\circ\sigma_k(t)|\,dt.
$$
Hence,  if we show that 
\begin{equation}\label{nonlip}
\displaystyle\lim_{k\to+\infty}\int_0^{S}|v_1\circ\varphi_{0_{k}}(r)-v_1\circ\varphi_{0}(r)|\varphi'_{0_k}(r)\,ds=0,
\end{equation}
 then there exists a subsequence of $(v_1-v_1\circ\varphi_{0}\circ\sigma_k)$ converging to 0 a.e.  on  $[0, T]$, and by the Dominated Convergence Theorem we obtain that, for such subsequence, 
\begin{equation}\label{nonlipg}
 \rho_{1_k}=\int_0^{T}\omega(|v(t)-v\circ\varphi_{0}\circ\sigma_h(t)|)\,dt\to0 \quad \text{as $k\to+\infty$,}
\end{equation}
so concluding the proof of \eqref{cunif}.

\noindent Since $|\varphi'_{0_k}|\le 1$, when $v_1$ is a continuous function (\ref{nonlip})  holds true  
owing to the uniform continuity of $v_1$ and to the uniform convergence of $\varphi_{0_k}$ to $\varphi_{0}$ on $[0,S]$.  For  $v_1\in L^1(T)$, $\forall\varepsilon>0$   there exists, by density, $\tilde v_1\in C_c([0,T], {\mathbb{R}}^l)$  such that $\int_0^{T}|\tilde v_1(t)-v_1(t)|\,dt\le\varepsilon.$ Hence we get
$$
\begin{array}{l}\int_0^{S}| v_1\circ\varphi_{0_k}(s)-v_1\circ\varphi_{0}(s)|\varphi'_{0_k}(s)ds\le 
\int_0^{S}| v_1\circ\varphi_{0_k}(s)-\tilde v_1\circ\varphi_{0_k}(s)|\varphi'_{0_k}(s)\,ds+\\ \, \\
\quad\int_0^{ S}|\tilde v_1\circ\varphi_{0_k}(s)-\tilde v_1\circ\varphi_{0}(s)|\varphi'_{0_k}(s)\,ds+
\int_0^ S|\tilde v_1\circ\varphi_{0 }(s)- v_1\circ\varphi_{0}(s)|\varphi'_{0_k}(s)\,ds.
\end{array}
$$
Performing the change of variable $t=\varphi_{0_k}(s)$,   the first integral on the r.h.s. is smaller than $\varepsilon$, while 
the second one converges to 0 because $\tilde v_1$ is continuous. For the third integral on the r.h.s., taking into account that  $v_1$ and $\tilde v_1|$ are bounded maps,  by the weak$^*$  convergence of  $\varphi'_{0_k}$ to   $\varphi'_0$  we derive that
  $$
 \int_0^{S}|\tilde v_1\circ\varphi_{0 }(s)- v_1\circ\varphi_{0}(s)|\varphi'_{0_k}(s)\,ds\to \int_0^{S}|\tilde v_1\circ\varphi_{0}(s)-v_1\circ\varphi_{0}(s)|\varphi'_{0}(s)\,ds 
$$
as $k\to+\infty$,  and the last term is smaller   than $\varepsilon$   by the change of variable  $t=\varphi_{0}(s)$.  By the arbitrariness of $\varepsilon>0$ this concludes the proof of (\ref{nonlip}). 
\end{proof}

\end{document}